\documentclass[12pt, letterside]{article}
\usepackage{amssymb, amsfonts}
\usepackage{amsmath}
\usepackage{amscd}
\usepackage[all]{xy}
\usepackage{amsthm}
\usepackage{enumerate}

\theoremstyle{plain}
\newtheorem{thm}[subsection]{Theorem}
\newtheorem{prop}[subsection]{Proposition}

\newtheorem{lem}[subsection]{Lemma}

\theoremstyle{definition}
\newtheorem{definition}[subsection]{Definition}

\newtheorem{rem}[subsection]{Remark}

\DeclareMathOperator{\Z}{\mathbb{Z}}

\DeclareMathOperator{\Q}{\mathbb{Q}}

\DeclareMathOperator{\Gal}{\text{Gal}}

\DeclareMathOperator{\OO}{\mathcal{O}}

\DeclareMathOperator{\DR}{\text{DR}}
\DeclareMathOperator{\Syn}{\mathcal{S}}
\DeclareMathOperator{\Tan}{\mathcal{T}}
\newcommand{\La}{{{\Lambda}}}

    \DeclareFontFamily{U}{wncy}{}
    \DeclareFontShape{U}{wncy}{m}{n}{<->wncyr10}{}
    \DeclareSymbolFont{mcy}{U}{wncy}{m}{n}
    \DeclareMathSymbol{\Sha}{\mathord}{mcy}{"58}

\numberwithin{equation}{section}

\title{Tamagawa number conjecture for a $p$-adic family of uniform $F$-crystals}
\author{Tung T. Nguyen}
\begin{document}
\maketitle
\begin{abstract} K. Kato has recently constructed certain syntomic complexes associated with a uniform $F$-crystal over a smooth project variety $X$ and related their cohomology groups to special values of the $L$-function attached to $F$. In this paper, we generalize his results to a $p$-adic family of $F$-crystals.
\end{abstract}
\section{Introduction and our main results}

One of the fascinating features of number theory is the deep connection between an object's arithmetic properties and the analytic properties of its L-function. Perhaps one of the most precise conjectures about this connection is the Tamagawa number conjecture discovered by S. Block and K. Kato in $\cite{[BK]}$. More precisely, for each motive $M$ over a global field $F$, S. Bloch and K. Kato defined some important arithmetic groups associated with $M$ and related them to special values of the $L$-function attached to $M$. While this conjecture is still wide open when $F$ is a number field, much is known when $F$ is a function field of one variable over a finite field $k$ of characteristics $p$. For example, when $M$ is the $\ell$-adic realization of a motive over $F$ and $\ell \neq p$, Bloch and Kato showed in $\cite{[BK]}$ that the Tamagawa number conjecture is a consequence of Grothendieck's formula (see proposition $5.21$ of $\cite{[BK]}$.) However, when $\ell = p$, their method does not work, and new ideas are needed.

Quite recently, K. Kato in $\cite{Fcrystals}$ treated the remaining cases of the Tamagawa number conjecture over function fields when $\ell=p$ using some previous work of A. Ogus on uniform $F$-crystals. More precisely, let $X$ be a smooth projective smooth variety over a finite field of characteristics $p$, and $(D, \Phi)$ be a uniform $F$-crystal over X. In $\cite{Fcrystals}$, Kato constructed some syntomic complexes associated with $F$ and expressed the special values of the $L$-function attached to $F$ in terms of the cohomology groups of these syntomic complexes (see \cite{Fcrystals}, proposition $5.4$). In this article, we generalize his results to a $p$-adic family of uniform $F$-crystals. More precisely, we have the following theorem (we refer the readers to the body of this article  for the definitions of the terms appeared here.) 

\begin{thm}
Let $X$ be a smooth projective variety over a finite field $k$ of characteristics $p$. Let $F=(D, \Phi)$ be a uniform $F$-crystal over $X$. Let $K_{\infty}$ be a $p$-adic extension of the function field $K$ of $X$ satisfying certain technical conditions mentioned in section $4$ and $X_n$ be a tower of covers of $X$ associated with this extension. Let $\Syn(D), \Tan(D)$ be the complexes associated with $F$ constructed in section 2. Under some technical assumptions, there exists an element $\mathcal{L}_D\in K_1(\La(G)_{S^*})$ such that 
$$\partial_G(\mathcal{L}_D)= [R\varprojlim_n R\Gamma(X_n, \Syn(D))] +[R\varprojlim_n R\Gamma(X_n, \Tan(D)))]$$
in $K_0(\frak M_H(G))$ where $\partial$ is the boundary map of the long exact sequence of $K$-groups (see $\cite{[CFKSV]}$, (24).)
\end{thm}

\section{Syntomic complexes and Tamagawa number conjecture for $F$-crystals}
Throughout this article, we assume that $ X $ is a smooth projective variety over a finite field $k$ of characteristics $p$. Let $F=(D, \Phi)$ be uniform $F$-crystals on $X$.  In this section, we provide a brief overview of Kato's construction of certain syntomic complexes assocaited with a uniform $F$ crystals over $X$. We refer the readers to $\cite{Fcrystals}$ for further details. 

\begin{rem}
One of the main reasons that we restrict our study to uniform $F$-crystals is that uniform $F$-crystals form a good category. They also have good integral properties which will allow us to define integral syntomic complexes. 

Most ``geometric'' $F$-crystals are uniform. More precisely, suppose $\pi: X \to Y$ is a proper smooth morphism of smooth schemes then under mild assumptions, the higher direct image $R^{m} \pi_{*} (\OO_{X})_{\text{cris}}$ is a uniform $F$-crystal on $X$ (for more details, see $\cite{[Mazur1]}, \cite{[Mazur2]}, \cite{[Ogus2]}.$)
\end{rem}

For each $q \in \Z$, let $\omega_{X}^{q} :=\Omega_{X}^{q}$ the sheaf of differential $q$-forms on $X$. Let $D_{X}$ be the vector bundle together with an integrable connection $\nabla$ on $X$ associated with $D$: 
\[ \nabla: D_{X} \to D_{X} \otimes_{\OO_{X}} \omega_{X}^{1} .\] 

Let $X \subset Y$ be a closed immersion of $X$ into a $p$-adic formal scheme $Y$ over $W(k)$. Let $D_{X}(Y)$ the PD envelop of $X$ in $Y$ and $\OO_{D_{X}(Y)}$ be the structure sheaf of the formal scheme $D_{X}(Y)$. Let $\omega_{Y}^{q}:=\Omega_{Y}^{q}$ be the sheaf of differential forms on $Y$. Let $D_{Y}$ be the locally free $\OO_{D_{X}(Y)}$-module together with an integrable connection $\nabla$ associated with $D$ which is defined by
\[ \nabla: D_{Y} \to D_{Y} \otimes_{\OO_{Y}} \omega_{Y}^{1} .\] 
For an $F$-crystal $(D, \Phi)$ we will use the same notation $\Phi$ for the induced isomorphism:
\[ \Phi: \Q_{p} \otimes_{\Z_p} F^{*} D_{Y} \to \Q_p \otimes_{\Z_p} D_{Y} .\] 
We remark that $\Phi$ is compatible with the integrable connection $\nabla$ endowed with $D_{Y}$.

To define syntomic complexes, we first need to introduce several filtration on $D_{Y}$ and $F^{*} D_{Y}$. 
\begin{definition} \label{fil}
For each $r \in \Z$, we define 
\[ N_{r}(D)_{Y} := D_{Y} \cap p^{-r} \Phi(F^{*}D_{Y}), \]
 \[ N^{r}(D)_{Y} :=D^{Y} \cap p^{r} \Phi(F^{*}D_{Y}), 
\]
and 
\[ M^{r}(D)_{Y} := \{x \in F^{*}D_Y| p^{-r} \Phi(x) \in D_{Y} \} \subset F^{*}D_Y .\] 
\end{definition}
First, we explain the construction of some de Rham complexes using the filtration on $D_{Y}$ and $F^{*} D_{Y}$ constructed  above. Following Deligne and Kato, we make the following convention. Let $C$ be a filtered complex. We define $\overline{C}$ to be the following complex  
\[ \overline{C}^{q}=\{x \in (^q C)^q| dx \in (^{q+1}C)^q \} .\] 
A key property of this construction is that if if $f: C \to C'$ is a homomorphism of filtered complexes such that for all $q$ $f: ^q C \to ^q C'$ is a quasi-isomorphism then the induced map $\overline{C} \to \overline{C'}$ is a quasi-isomorphism.

Recall that we have the following de Rham complex associated with $D$:
\[ \DR(D)_{Y}=[ D_{Y} \xrightarrow{\nabla} D_{Y}\otimes_{\OO_{Y}} \omega_{Y}^{1} \xrightarrow{\nabla} D_{Y} \otimes_{\OO_Y} \omega_{Y}^2 \xrightarrow{\nabla} \ldots ]\]
Using the filtration on $D_{Y}$ and $F^{*} D_{Y}$ defined in $\ref{fil}$, we define the following complexes. 
\[ N_{r} \DR(D)_{Y}=\overline{C}, \text{where} \quad ^q C=\DR(N_{r-q}(D))_{Y},\]
\[ N^{r} \DR(D)_{Y}=N_{-r} \DR(D)_{Y}, \]
and 
\[ M^r \DR(D)_{Y}=\overline{C}, \text{where} \quad ^q C =p^{q} \DR(M^{r-q}(D))_{Y} .\]
We remark that
\[  N_{(-\infty)} \DR(D)_{Y}=N^{(\infty)} \DR(D)_Y=p^{-r} N^r \DR(D)_{Y} ,\]
is independent of $r \gg 0.$

We define the syntomic complex $\Syn(D)_{Y}$ as the mapping fiber of the map \[ 1-\Phi \eta: N_0 \DR(D)_{Y} \to N_{(-\infty)} \DR(D)_{Y} .\] 
Here  $1$ is the inclusion map $N_0 \DR(D)_{Y} \to N_{(-\infty)} \DR(D)_{Y}$ and $\Phi \eta$ is the map defined as follow. First, $\eta: D_{Y} \to F^{*} D_{Y}$ is the map sending $x \in D_{Y}$ to $\eta(x)=1 \otimes x \in F^{*}_{D_{X}(Y)}D_{Y}=F^{*}D_{Y}.$ We then define $\Phi \eta$ as the induced map from $\Q_p \otimes_{\Z_p} \DR(D)_{Y} \to \Q_p \otimes_{\Z_p} \DR(D)_{Y}$ whose degree $q$-part is given by 
\[ x \otimes \omega \mapsto \Phi(\eta(x)) \otimes F_{Y}(w) .\] 
The constructions of the complexes $N_0 \DR(D)_{Y}, N_{(-\infty)} \DR(D)_{Y}, \Syn(D)_{Y}$,  are local in nature. However, we can glue these local constructions to get global complexes $\Syn(D)$ in the derived category of sheaves of abelian groups on the etale site of $X$. In other words, we have the following distinguished triangle. 
\[ \Syn(D) \to N_0 \DR(D) \to N_{(-\infty)} \DR(D) \to \Syn(D)[1] .\]
We also define $\Tan(D)$ as the complex obtained by gluing the following local complexes
\[ \Tan(D)_{Y}=N_{(-\infty)} \DR(D)_{Y}/N_0 \DR(D)_{Y} .\] 
As above, we have the following distinguished triangle. 
\[ N_0 \DR(D) \to N_{(-\infty)} \DR(D) \to \Tan(D) .\]
With these preparations, we  are now ready to state the main theorem of $\cite{Fcrystals}$.
\begin{prop} (\cite{Fcrystals}, proposition $5.4$) \label{L-function}
Suppose that $1-\Phi: H^m(X, \DR(D)_{\Q_p} \to H^m(X, \DR(D)_{\Q_p})$ are isomorphisms for all $m$. Then, the complexes $\Syn(D)$ and $\Tan(D)$ over the etale site of $X$ have finite cohomology groups. Let  $[R \Gamma(X, \Syn(D))]$ and $[R \Gamma(X,\Tan(D))]$ be the corresponding classes in $K_0(\Z_p, \Q_p)$. Then, we have
\[ \partial{L(D,1)}=[R \Gamma(X, \Syn(D))]+[R \Gamma(X,\Tan(D))] .\] 
Here $\partial$ is the boundary map of the long exact sequence of $K$-groups (see $\cite{[CFKSV]}$, (24).) 
\end{prop}
\section{Tamagawa number conjecture for a $p$-adic family of $F$-crystals, Preparations}
In this section, we generalize proposition $\ref{L-function}$ to the arithmetic $\Z_p$-extension of $X$. More precisely, by Galois theory, there exists a $p$-adic Galois extension $k_{\infty}/k$ such that $\Gal(k_{\infty}/k)=\Gamma=\Z_p$. Let us denote by $\Gamma_n$ the abelian group $\Z/p^n \Z$ and by $k_n$ the unique subfield of $k_{\infty}$ such that $\Gal(k_n/k)=\Gamma_n$. We call $k_n$ the $n$-th layer of $k_{\infty}/k$. Let $X_n$ be the variety $X \otimes k_n$. We define the following total complexes \[ I_n= R \Gamma(X_n, N_{0} DR(D)),\]
\[ P_n =R \Gamma(X_n, N_{-\infty} DR(D)),\]
\[ N_n=R \Gamma(X_n, \Syn(D)), \]
\[ L_n = R \Gamma(X_n, \Tan(D)) .\] 

We also define $N_{\infty}$ be the inverse limits of $(N_n)$ 
\[ N_{\infty}= R \varprojlim_{n} N_n. \] 
For each $S \in \{L,I,P \}$, we define $S_{\infty}$ similarly.  

For $S \in \{N, L, I,P \}$, $S_n$ is a natural module over $\Z_p[\Gamma_n]$. Consequently, $S_{\infty}$ is a module over the Iwasawa algebra $\Lambda(\Gamma)$ where 
\[ \Lambda(\Gamma)=\varprojlim_{n} \Z_p[\Gamma_n] .\] 
We denote by $\underline{N}=(N_n)$ the object in the derived category $D^{b}(_{\underline{\Gamma}}(\Z_p-mod))$ of normic systems along the profinite group $\Gamma$ (see section 2 of $\cite{[TV]}$ for the precise definition of normic systems.) For each $S \in \{L,I,P \}$, we define $\underline{S}$ similarly.

By construction, we have the following proposition (see proposition $5.1$ of $\cite{[TV]}$ for a similar statement.) 
\begin{prop} \label{normic}
Let $\underline{W}=(W(k_n)) \in _{\underline{\Gamma}}(\Z_p-mod)$ be the natural normic system of $\Z_p$-modules along $\Gamma$. For $X \in \{I, P, L \}$, there is a canonical isomorphism in $D^{b}(_{\underline{\Gamma}}(\Z_p-mod))$:
\[ \underline{W} \bigotimes^{L} X_{0} \cong \underline{X}. \]
\end{prop}
\begin{rem}
Note that this is note true for $\underline{N}$. 
\end{rem}

We have the following result, which is a natural generalization of proposition $\ref{L-function}$ to a family of $F$-crystals. 

\begin{prop} \label{first_theorem}
$N_{\infty}$ and $L_{\infty}$ are torsion over $\Lambda(\Gamma)$. Furthermore, there exists an element $\mathcal{L}_{D} \in K_1(\Lambda(\Gamma))$ such that 
\[ \partial(\mathcal{L}_{D}) = [N_{\infty}]+[L_{\infty}], \] 
where $[N_{\infty}]$ and $[L_{\infty}]$ are the corresponding classes of $N_{\infty}$ and $L_{\infty}$ in $K_0(\Lambda(\Gamma))$ and $\partial$ is the boundary map defined in $\cite{[CFKSV]}$, (24). 

\end{prop}
We provide a proof for the first statement. The second statement will be proved in a more general setting discussed in the next section. Our proof is based on the ideas of the paper $\cite{[TV]}$ and we are thankful to its authors of for the innovative ideas. We refer the readers to section 2 and section 3 of that paper for some of the notions that we use here.

\begin{proof}
By the above construction, we have the following distinguished triangles in $D(\Lambda(\Gamma))$:

\[ \Q_p \otimes_{\Z_p} N_{\infty} \to \Q_p \otimes_{\Z_p} I_{\infty} \xrightarrow{1-\phi} \Q_p \otimes_{\Z_p} P_{\infty} \to N_{\infty}[1], \] 

and 
\[ \Q_p \otimes_{\Z_p} I_{\infty} \xrightarrow{\mathbf{1}} \Q_p \otimes_{\Z_p} P_{\infty} \to \Q_p \otimes_{\Z_p} L_{\infty} \to \Q_p \otimes_{\Z_p} I_{\infty}[1].\] 
By proposition $\ref{normic}$, we can rewrite these triangles as 
\begin{equation} \label{syntomic}
\Q_p \otimes_{\Z_p} N_{\infty} \to W_{\infty} \otimes_{\Z_p}^{L} (\Q_p \otimes_{\Z_p} I_0) \to W_{\infty} \otimes_{\Z_p}^{L} (\Q_p \otimes_{\Z_p} P_0) \to \Q_p \otimes_{\Z_p} N_{\infty}[1] ,
\end{equation}
and 
\[  W_{\infty} \otimes_{\Z_p}^{L} (\Q_p \otimes_{\Z_p} I_0) \xrightarrow{\mathbf{1}} W_{\infty} \otimes_{\Z_p}^{L} (\Q_p \otimes_{\Z_p} P_0) \to W_{\infty} \otimes_{\Z_p}^{L} (\Q_p \otimes_{\Z_p} L_0) \to  W_{\infty} \otimes_{\Z_p}^{L} (\Q_p \otimes_{\Z_p} I_0)[1].\]

Because $L_{0}$ is $\Z_p$-torsion, we have $\Q_p \otimes_{\Z_p} L_{0}=0$. Consequently, $\mathbf{1}$ gives an isomorphism 
\[ W_{\infty} \otimes_{\Z_p}^{L} (\Q_p \otimes_{\Z_p} I_0) \cong W_{\infty} \otimes_{\Z_p}^{L} (\Q_p \otimes_{\Z_p} P_0). \]

Let us denote by $\iota$ the inverse of the the isomorphism $\mathbf{1}.$ Because $\Q_p$ and $W_{\infty}$ are flat over $\Z_p$, the long exact sequence associated with the triangle $\ref{syntomic}$ can be written as
\[ \ldots \to \Q_p \otimes_{\Z_p} H^{i}(N_{\infty}) \to W_{\infty} \otimes_{\Z_p}^{L} (\Q_p \otimes_{\Z_p} H^{i}(P_{\infty})) \xrightarrow{1-\phi \iota} W_{\infty} \otimes_{\Z_p}^{L} (\Q_p \otimes_{\Z_p} H^{i}(P_{\infty})) \to \ldots \]

By lemma $5.3$ of $\cite{[TV]}$, the map $1-\phi \iota$ is injective. Consequently, the above sequence gives rise to the following short exact sequence 
\begin{equation} \label{short_exact_sequence}
0 \to \Q_p \otimes_{\Z_p} H^{i+1} (P_{\infty}) \xrightarrow{1-\phi \iota} \Q_p \otimes_{\Z_p} H^{i+1} (P_{\infty}) \to \Q_p \otimes_{\Z_p} H^{i}(N_{\infty}) \to 0.
\end{equation} 

Let us denote by $Q(\Gamma)$ the total quotient field of $\Lambda(G)$. By applying $Q(\Gamma) \otimes_{\Lambda(\Gamma)} (-)$ to the short exact sequence $\ref{short_exact_sequence}$, we have 
\[ 0  \to Q(\Gamma) \otimes_{\Lambda(\Gamma)} H^{i}(P_{\infty}) \xrightarrow{1-\phi \iota} Q(\Gamma) \otimes_{\Lambda(\Gamma)} H^{i}(P_{\infty}) \to Q(\Gamma) \otimes_{\Lambda(\Gamma)} H^{i+1}(N_{\infty}) \to 0 .\] 
Because $H^{i}(P_{\infty})$ is a finitely generated $\Lambda(\Gamma)$-module, $Q(\Gamma) \otimes_{\Lambda(\Gamma)} H^{i}(P_{\infty})$ is a finite dimensional vector space over $Q(\Gamma)$. Consequently, the map 
\[ Q(\Gamma) \otimes_{\Lambda(\Gamma)} H^{i+1}(P_{\infty}) \xrightarrow{1-\phi \iota} Q(\Gamma) \otimes_{\Lambda(\Gamma)} H^{i+1}(P_{\infty}) \]
must be an isomorphism. We conclude that $Q(\Gamma) \otimes_{\Lambda(\Gamma)} H^{i+1}(N_{\infty})=0$ for all $i$, and hence $N_{\infty}$ is torsion over $\Lambda(\Gamma)$.

\end{proof}
\section{Tamagawa number conjecture for a $p$-adic family of $F$-crystals}
In this section, we provide a generalization of proposition $\ref{first_theorem}$. We will follow the notation in  $\cite{[CFKSV]}$. 

Let $K$ be the function field of $X$. Let $k$ be the total constant field of $K$ and let $k_{\infty}$ be the unique $\Z_p$-extension of $k$ introduced in section $2.4$.  Let $K_{\infty}$ be a Galois extension of $K$ satisfying the following conditions (i)--(iii). 
\begin{enumerate}[(i)]
\item  The extension $K_{\infty}/K$ is unramified at every point of $X$. 
\item  $K_{\infty}\supset K k_{\infty}$.
\item  Let $G:=\Gal(K_{\infty}/K)$. Then $G$  is a $p$-adic Lie group having no element of order $p$.
\end{enumerate}

Let $H:=\Gal(K_{\infty}/Kk_{\infty})$, $\Gamma:= G/H=\Gal(Kk_{\infty}/K)$. Let $\La(G)=\Z_p[[G]]$ be the completed group ring of $G$ and define $\La(H)$ and $\La(\Gamma)$ similarly.

 Take finite Galois extensions $K_n$ of $K$ in $K_{\infty}$ ($n\geq 1$) such that $K_n\subset K_{n+1}$ and $K_{\infty} =\cup_n\; K_n$. Let $X_n$ be the integral closure of $X$ in $K_n$, so $X_n$ is a finite \'etale Galois covering of $X$. This $X_n$ is a natural generalization of $X_n$ of \S2.4.

Let $S$ be the set of all elements $f$ of $\Lambda(G)$ such that $\Lambda(G)/\Lambda(G)f$ are finitely generated $\Lambda(H)$-modules. Let $S^*= \cup_{n\geq 0}\; p^nS$. 
Then by \cite{[CFKSV]} Theorem 2.4, $S^*$ is a multiplicatively closed left and right Ore set in $\Lambda(G)$ and all elements of $S^*$ are non-zero-divisors of $\Lambda(G)$. Hence we have the ring of fractions $\Lambda(G)_{S^*}$ by inverting all elements of $S^*$ and the canonical homomorphism $\Lambda(G)\to \Lambda(G)_{S^*}$ is injective. 

As in $\cite{[CFKSV]}$, let $\frak M_H(G)$ be  the category of finitely generated $S^*$-torsion $\La(G)$-modules. If $M$ is a finitely generated $\Lambda(G)$-module and $M(p)$ denotes the $\Lambda(G)$-submodule of $M$ consisting of all elements killed by some powers of $p$, $M$ belongs to $\frak M_H(G)$ if and only if $M/M(p)$ is finitely generated as a $\Lambda(H)$-module.

We have a long exact sequence of $K$-groups (see $\cite{[CFKSV]}$, (24))
$$\dots \to K_1(\La(G)) \to K_1(\La(G)_{S^*}) \overset{\partial_G}\to K_0(\frak M_H(G)) \to K_0(\La(G))\to \dots.$$

We have the following lemma.
\begin{lem}\label{MHGlem}
Let $D=(D, \Phi)$ be a uniform $F$-crystal on $X$. We have 
\begin{enumerate}[(i)]

\item $R\varprojlim_n R\Gamma(X_n, N_0DR(D))$ and $R\varprojlim_n R\Gamma(X_n, N_{-\infty}DR(D))$ are bounded complexes and their cohomology groups are finitely generated $\La(G)$-modules. 

\item $R\varprojlim_n R\Gamma(X_n, \Tan(D))$ is a bounded complex and its cohomology groups belong to $\frak M_H(G)$.

\item $R\varprojlim_n R\Gamma(X_n, \Syn(D))$ is a bounded complex. 

\item Assume that there is a finite extension $K'$ of $K$ in $K_{\infty}$ such that $\Gal(K_{\infty}/K')$ is pro-$p$ and such that if $X'$ denotes the integral closure of $X$ in $K'$ and $\Gamma'$ denotes $\Gal(K'k_{\infty}/K')$, then the $\mu$-invariants of the cohomology groups $R\varprojlim_n R\Gamma(X' \otimes_k k_n, \Syn(D))$, which are finitely generated  torsion $\Lambda(\Gamma')$-modules by Proposition 13,  is zero. Then the cohomology groups of 
$R\varprojlim_n R\Gamma(X_n, \Syn(D))$ belong to $\frak M_H(G)$. 

\end{enumerate} 
\end{lem}
\begin{proof} Let $K'$ be as in the assumption of (iv). Let $G'=\Gal(K_{\infty}/K')$, $H'=\Gal(K_{\infty}/K'k_{\infty})$. These are pro-$p$ groups. Let $X'$ be the integral closures of $X$ in $K'$.

 By Theorem 2.11 of $\cite{[TV]}$, we have \[ \Z_p\otimes_{\La(G')}^L  R\varprojlim_nR\Gamma(X_n, N_0DR(D)) = R\Gamma(X', N_0DR(D)).\]
 The right hand side is a finitely generated $\Z_p$-module. Hence (i) follows from  theorem 2.11 of $\cite{[TV]}$ and Nakayama's lemma applied to the local ring $\La(G')$ (this is a local ring because  $G'$ is a pro-$p$ group).
 The proof for $N_{-\infty}DR(D)$ is similar. 
 (ii) and (iii) follow from (i) by the distinguished triangles 

\begin{enumerate}[(1)]
\item  $R\varprojlim_n R\Gamma(X_n, \Syn(D)) \to R\varprojlim_n R\Gamma(X_n, N_0DR(D)) \to R\varprojlim_n R\Gamma(X_n, N_{-\infty}DR(D)) \to$,

\item $R\varprojlim_n R\Gamma(X_n, N_0DR(D)) \to R\varprojlim_n R\Gamma(X_n, N_{-\infty}DR(D))\to R\varprojlim_n R\Gamma(X_n, \Tan(D)) \to$

\end{enumerate}

By the assumption of (iv),  cohomology groups of $R\varprojlim_n R\Gamma(X' \otimes_k k_n, \Syn(D))$ are finitely generated $\Z_p$-modules. By theorem $2.11$ of $\cite{[TV]}$, we have also
\[\Z_p \otimes_{\La(H')} R\varprojlim_n R\Gamma(X_n, \Syn(D)) = R\varprojlim_n R\Gamma(X' \otimes_k k_n, \Syn(D)). \] 

By Nakayama's lemma applied to the local ring $\La(H')$, we have that all cohomology groups of $R\varprojlim_n R\Gamma(X_n, \Syn(D))$ are finitely generated $\La(H')$-modules and hence finitely generated $\La(H)$-modules. This proves (iv).  
\end{proof}

In the rest of this \S2.5, we assume that the assumption of \ref{MHGlem} (iv) is satisfied. 

We define the $p$-adic $L$-function $\mathcal{L}_D\in K_1(\La(G)_{S^*})$ as the minus of the class of the automorphism 
$1-\phi\iota$ of $\La(G)_{S^*}\otimes_{\La(G)}  R\varprojlim_n R\Gamma(X_n, N_{-\infty}DR(D))$.  

The special values of $\mathcal{L}_D$ in the sense of $\cite{[CFKSV]}$ \S3 are $L$-values of $D$:
$$\rho(\mathcal{L}_D)=   L(X, D, \rho^{\vee}, 0)$$
for every finite dimensional continuous $p$-adic representation $\rho$ of $G$ which factors through a finite quotient of $G$. This is reduced to proposition 5.4 of $\cite{Fcrystals}$ by the proof of $\cite{[TV]}$ Theorem 1.1 (ii).

By the above distinguished triangles (1) and (2), we have (a version of Iwasawa main conjecture)
\begin{prop}
$$\partial_G(\mathcal{L}_D)= [R\varprojlim_n R\Gamma(X_n, \Syn(D))] +[R\varprojlim_n R\Gamma(X_n, \Tan(D)))]$$
in $K_0(\frak M_H(G))$. 
\end{prop} 

\section*{Acknowledgements}

I am grateful to my advisor, Professor Kazuya Kato, for introducing me to the Tamagawa number conjecture, sharing his insights, and his constant support during the elaboration of this work. As it is visible to the readers, our article is strongly influenced by the work $\cite{[TV]}$ of D. Trihan and F. Vauclair. I am thankful to these authors for their innovative ideas. 

% Figures and tables, if you decide to leave them to the end
%\input{figure}
%\input{table}

\end{document}